\documentclass[10pt,a4paper]{amsart}
\usepackage{graphicx}

\usepackage{amsmath}%
\usepackage{amssymb,indentfirst,calc,mathrsfs,euscript,bbm}%
\usepackage{dsfont}

\usepackage{setspace}%
\usepackage[reals]{layout}%
\usepackage{xr}%
\usepackage{amscd}%
\usepackage{comment}
\usepackage{epstopdf}
\usepackage{color}

 \def\E{{ \mathsf E}} 

 \def\E{{ \mathsf E}} 

\newtheorem{theorem}{Theorem}
\newtheorem{proposition}[theorem]{Proposition}
\newtheorem{lem}[theorem]{Lemma}

\theoremstyle{definition}
\newtheorem{defi}[theorem]{Definition}

\theoremstyle{remark}
\newtheorem{remark}[theorem]{Remark}

\DeclareMathOperator{\Argmax}{Argmax}

\begin{document}
\thanks{This research was supported by grant ANR-10-BLAN 0112 (France)}
\title{A zero-sum stochastic game with compact action sets and no
asymptotic value}
\author{Guillaume Vigeral}
\address{Universit\'e Paris-Dauphine, CEREMADE, Place du Mar{\'e}chal De Lattre de Tassigny. 75775 Paris cedex 16, France vigeral@ceremade.dauphine.fr}

%

%
%

\date{}

%


\begin{abstract}
We give an example of a zero-sum stochastic game with four states, compact action sets for each player, and continuous payoff and transition functions, such that the discounted value does not converge as the discount factor tends to 0, and the value of the $n-$stage game does not converge as $n$ goes to infinity.
\end{abstract}

\maketitle

\section{Introduction}

Two person zero-sum stochastic games have been widely studied since Shapley introduced them in \cite{Sh}. They model interactions repeated in discrete time between two players with opposite interests. The state of nature evolves as a function of the current state and of the actions chosen by each player, and determines which zero-sum game the players are facing at each time period. Hence, the actions of the players have an influence both on the payoff today and on the law of the state of nature tomorrow.

There are several ways of evaluating the payoff of such a stochastic game. For any integer $n\in\mathds{N}$, one defines the $n-$stage game for which Player 1 (resp. Player 2) maximizes (resp. minimizes) his average gain on the first $n$ stages. For any $\lambda\in]0,1]$, one defines the $\lambda$-discounted game\footnote{In which the future has weight $1-\lambda$ ; we warn the reader that in the literature the opposite convention $\delta=1-\lambda$ is often used.} for which Player 1 (resp. Player 2) maximizes (resp. minimizes) his $\lambda-$discounted payoff. Some of the main questions in the theory of zero-sum stochastic games are related to the asymptotic behavior of the values of these games as players grow more and more patient:
\begin{itemize}
\item Does the value of the $n-$stage game converge as $n$ tends to infinity ?
\item Does the value of the $\lambda-$discounted game converge as $\lambda$ tends to 0 ?
\item Are the two limits equal ?
\end{itemize}
When the answers to these three questions are positive, the game is said to have an asymptotic value. A nice explanation of why the asymptotic value should exist for games regular enough is the following \cite{So2}. An $n-$stage game can be seen as a game played in the time interval $[0,1]$, where the payoff is $\int_0^1 g_t$, and in which the players only moves at time $\frac{k}{n}$. Similarly, in a $\lambda$-discounted game, they only play at time $\lambda$, $\lambda+\lambda(1-\lambda)$, and so on. As $n$ goes to infinity and $\lambda$ goes to 0, these games can thus be viewed as some time discretizations of an hypothetic game played in continuous time on $[0,1]$, and thus the values should converge to the value of this "limit game".

Stochastic games were first studied in the case of a finite number of states and when each player has only finitely many actions. Existence and characterization of the values for a fixed $\lambda$ or $n$ is due to Shapley \cite{Sh} and relies on von Neumann's minmax theorem \cite{voNe} as well as Banach's fixed point theorem. In this framework, asymptotic value was established first for recursive \cite{Ev} and absorbing games \cite{Koh}, then in general (see \cite{BeKo,BeKo2} for the original proof using Tarski-Seidenberg's Theorem, or \cite{Ol} for a recent proof involving linear programming).

Since minmax theorems also hold true for games with compact action sets and continuous payoffs \cite{Si}, the values exist \cite{MaPa} for fixed $n$ or $\lambda$
for games with finitely many states, compact action sets for each player, and continuous payoff and transitions. In this framework, asymptotic value was established for recursive \cite{So4,SoVi} and absorbing \cite{RoSo,SoVi} games
, and was conjectured to hold true in general \cite{So2}. 

Let us mention that the existence of an asymptotic value was established in the framework of Markov decision processes and dynamic programming \cite{Ba1,Ba2,Ba3,DyYu,Ren3} ; for games with incomplete information \cite{AuMa,CLS,MeZa,RoSo} ; as well as for some stochastic games with incomplete information \cite{Ren, Ren2,Ro,RoVi}.

In this paper we answer by the negative to the conjecture in \cite{So2} by constructing a game with four states, compact action sets, and continuous payoff and transitions, whose values do not converge as $n$ tends to infinity or $\lambda$ tends to 0. Surprisingly, it is possible to construct a compact game in which Player $1$ can guarantee a payoff of $1$ in any $10^{2k}-$stage game, while Player 2 can guarantee a payoff of $-1$ in any $10^{2k+1}-$stage game. The idea of the counterexample is to construct transition functions that are continuous but oscillate infinitely often.  These oscillations of the transition functions yield oscillations - and thus divergence - of the values.

The paper is structured as follows. The first section gives the model of compact stochastic games and define discounted and finitely repeated values. The next section is the main one in which some counterexamples are constructed: first we give some examples in which the discounted value diverges, then we show that the value of the $n-$stage game diverges as well for some of these examples. The last section gives some concluding remarks as well as some open questions.

\section{Model}
A \emph{compact} two person zero-sum stochastic game $\Gamma$ is defined by a finite state space  $\Omega$,  compact metric action spaces  $I$ and  $J$ for Player 1 and 2  (we will denote the mixed actions sets of Player 1 and Player 2  $X =\Delta (I) $ and  $Y= \Delta (J)$, respectively\footnote{For a compact metric space $K$, $\Delta (K)$ denotes the set of Borel probabilities on $K$, endowed with the weak-$\star$ topology.} ), a jointly continuous real bounded payoff $g$ on $I \times J \times\Omega $ and a jointly continuous  transition  $\rho$ from $I \times J \times\Omega $ to $\Delta (\Omega)$. When $I$ and $J$ are finite the game is said to be \emph{finite}.

The game is played in discrete time. The initial state $\omega_1\in\Omega$ is known by both players. At stage $t$, given the state $\omega_t$, the players independently choose mixed moves $x_t \in X$ and $y_t \in Y$. The stage actions $i_t$ and $j_t$ are drawn according to $x_t$ and $y_t$ respectively. The stage payoff is $g_t = g( i_t, j_t, \omega_t)$, the new state  $\omega_{t+1}$  is selected according to $ \rho( i_t, j_t, \omega_t)$, and $(i_t,j_t,\omega_{t+1})$ is announced to the players.

We are mainly interested in discounted games: for any discount factor $\lambda \in ]0, 1]$, the $\lambda$-discounted game with initial state $\omega_1$ is denoted $\Gamma_\lambda(\omega_1)$; in this game Player 1 (resp. Player 2) maximizes (resp. minimizes) the expectation of $\sum_{t=1}^{\infty} \lambda ( 1 - \lambda)^{t-1}g_t$. The game $\Gamma_\lambda(\omega_1)$ has a value denoted by $v_\lambda(\omega_1)$, and one proves (see \cite{Sh} in the finite case and \cite{MaPa} in the compact one) that the function $v_\lambda:\Omega\to\mathds{R}$ is the only fixed point of the following equation:

\begin{eqnarray}
f(\omega)&=&\min_Y \max_X \left\{\lambda g(x,y,\omega)+ (1-\lambda)\E_{\rho(x,y,\omega)} f(\cdot)\right\}\label{eqminmax}\\
&=&\max_X \min_Y  \left\{\lambda g(x,y,\omega)+ (1-\lambda)\E_{\rho(x,y,\omega)} f(\cdot)\right\},\label{eqmaxmin}
\end{eqnarray}
where $g$ and $\rho$ are bilinearly extended to $X\times Y$, and the permutation of $\min$ and $\max$ is possible according to Sion's theorem\cite{Si}.

The following lemma gives an interesting sufficient condition for a function to be equal to $v_\lambda$.
\begin{defi}
A mixed action $x\in X$ (resp. $y\in Y$) is equalizing for the function $f$ in $\Gamma_\lambda(\omega)$ if for every $y\in Y$ (resp. every $x\in X$),
\[
f(\omega)=\lambda g(x,y,\omega)+ (1-\lambda)\E_{\rho(x,y,\omega)} f(\cdot).
\]
\end{defi}

\begin{lem}\label{equalizinglemma}
Let $\lambda\in]0,1]$ and assume that there exists a function $f$ such that for any state $\omega$, both players have an equalizing action in $\Gamma_\lambda(\omega)$. Then $f=v_\lambda$.
\end{lem}

\begin{proof}
Such an $f$ is a fixed point of (\ref{eqminmax}) and (\ref{eqmaxmin}), and $v_\lambda$ is the unique fixed point of these equations.
\end{proof}

The finitely repeated stochastic game with horizon $n$ and initial state $\omega_1$ is the game in which Player 1 (resp. Player 2) maximizes (resp. minimizes) the expectation of $\sum_{t=1}^{n} \frac{1}{n} g_t$. Its value is denoted by $v_n(\omega_1)$.

A compact stochastic game is said to have an \emph{asymptotic value} if $v_\lambda$ and $v_n$ converge (as $\lambda$ goes to 0 and $n$ to infinity respectively) and if the limits are the same.

In the next section we construct a compact stochastic game such that neither $v_\lambda$ nor $v_n$ converges. Hence there exists a compact stochastic game with no asymptotic value.
\section{Main section}

The main result of the paper is:

\begin{theorem}\label{maintheorem}
There exists a stochastic game with 4 states, in which the action sets are real intervals, the payoff and transition functions are continuous, and for which neither $v_\lambda$ nor $v_n$ does converge.
\end{theorem}

The remainder of this section is dedicated to the proof of this theorem.
\subsection{The intuition behind the construction}
Before going to the explicit construction of a counterexample we give some intuition about it and exhibit a subclass of compact games that is likely to contain a counterexample (if such a counterexample exists). We would like for this class to be as small as possible, in order to be more likely to find a precise counterexample within.

First, recall that a compact absorbing game has an asymptotic value \cite{RoSo}, so in any counterexample there must be at least two nonabsorbing states, and we consider the simplest case in which there are exactly two. Since in any compact stochastic game $v_\lambda(\cdot)$ converge for at least two initial starting states\cite{KoNe,Ne}, there must be at least four states. To make things simpler we may as well assume that the states for which $v_\lambda(\cdot)$ converges are absorbing, with different payoffs (else our game would be equivalent to a three states game), say $-1$ and $1$.

We also remark that, in compact games, it is the transitions functions, rather than the payoff functions, that are most likely be a source of oscillations of the values $v_\lambda$. A small variation of $g$ induces a small variation of $v_\lambda$; it is not the case for small variations of $\rho$. So, once again to simplify as much as possible, we assume that the payoff does not depend on the actions played by the player. Since compact recursive games have an asymptotic value \cite{So4}, the payoff in the two nonabsorbing states must be different, say $1$ and $-1$. 

It remains to understand which transition functions are likely to be problematic. First of all, we argue that under optimal play in $\Gamma_\lambda$, the absorption probability in each stage should be of the order of $\lambda$. Indeed, if it was much smaller than $\lambda$, then absorption would happen when the game has almost ended (that is, when the remaining part of the discounted payoff is negligible), so the absorbing states would be irrelevant and we might as well remove them. This would give us less than four states and thus an asymptotic value. On the other hand, if it was much greater than $\lambda$, absorption would occur almost immediately and the same play would give the same payoff for all small $\lambda$.

Similarly, we claim that the order of transition from one nonabsorbing state to the other should be on the order of $\lambda^\alpha$, for some $\alpha$ in $]0,1[$: if smaller it would almost never happen before absorption ; and if higher it would happen so often that the two states would be essentially the same, leaving us with a three states game and an asymptotic value.

Before considering compact games, we are first going to briefly study some finite games having all these features, to understand why $v_\lambda$ converge in the finite case and might not in the compact one. In fact, it turns out that such a game was already studied\footnote{Interestingly, this game was, at the time, a potential example of a finite game with no uniform value. In their example the payoff does depend on the chosen actions but this is irrelevant as it won't change the asymptotics of the optimal play.} by Bewley and Kohlberg (\cite{BeKo3} page 120). We make the following slight generalization\footnote{Their example is the particular case of $p^*_+=p^*_-=1$.}: consider the following family of finite stochastic games, where $p^*_+$ and $p^*_-$ are two parameters in $[0,1]$.
\begin{itemize}
\item There are two nonabsorbing states $\omega_+$ and $\omega_-$, and two absorbing states $1^*$ and $-1^*$.
\item Both players have two pures actions, Stay and Quit.
\item The payoff in each state is independent of the actions: it is $1$ in $\omega_+$ and $1^*$ ; $-1$ in $\omega_-$ and $-1^*$.
\item The transitions are given by the following matrices:
\end{itemize}

\vspace{.5cm}
\begin{center}
\begin{tabular}{lr}
\begin{tabular}{r|c|c|}
$\omega_+$&Stay&Quit\\  \hline Stay&$\omega_+$&$\omega_-$\\
\hline
Quit&$\omega_-$&$p_+1^*+(1-p^*_+)\omega_+$ \\
\hline
\end{tabular}
&
\begin{tabular}{r|c|c|}
$\omega_-$&Stay&Quit\\  \hline Stay&$\omega_-$&$\omega_+$\\
\hline
Quit&$\omega_+$&$p_-{-1}^*+(1-p^*_-)\omega_-$ \\
\hline
\end{tabular}
\end{tabular}
\end{center}
\vspace{.3cm}

Calculations show that:
\begin{itemize}
\item $\lim v_\lambda=v$  with $v(\omega_+)=v(\omega_-)=\frac{\sqrt{p^*_+}-\sqrt{p^*_-}}{\sqrt{p^*_+}+\sqrt{p^*_-}}$.
\item Optimal mixed actions in $\Gamma_\lambda$ are given, for  $k\in\{+,-\}$, by $x_\lambda(\omega_k)=y_\lambda(\omega_k)\approx\frac{\sqrt{\lambda}}{\sqrt{p^*_k}}$ as $\lambda$ goes to 0 (we identify a mixed action with the probability assigned to $Q$).
\end{itemize}

Recall that in any one-shot zero-sum game, if an optimal action of a player is completely mixed, any optimal action of the other player is equalizing. Thus, since both $x_\lambda$ and $y_\lambda$ are completely mixed, they are both equalizing in $\Gamma_\lambda$.

Taking the mixed extension of this finite game we get a compact game $\Gamma^c$. The (now pure) action $x_\lambda$ and $y_\lambda$ are optimal in $\Gamma^c_\lambda$. Since we want to discuss the influence of the parameters of the game on the transitions under optimal play, it is convenient to relabel the actions so that the optimal action of a player in $\Gamma_\lambda$ depends only on $\lambda$ and not on $p^*_+$ and $p^*_-$. For any nonabsorbing $\omega$, it can be shown that $x_\lambda(\omega_k)=y_\lambda(\omega_k)$ is decreasing for $\lambda$ small enough. Hence, by some suitable change of variables for the actions of each player in each state we get a compact game such that the stationary strategy $\lambda$ in each state is optimal (and equalizing) for each player. We have thus constructed a compact game such that:

\begin{itemize}
\item There are two nonabsorbing states $\omega_+$ and $\omega_-$, and two absorbing states $1^*$ and $-1^*$.
\item The set of actions of each player is $[0,1]$.
\item In each state, for each player, the pure action $\lambda$ is equalizing in $\Gamma_\lambda$ for $\lambda$ small enough.
\item $\rho(\omega_-|i,j,\omega_+)\approx \frac{\sqrt{i}+\sqrt{j}}{\sqrt{p^*_+}}$ ; $\rho(\omega_+|i,j,\omega_-)\approx \frac{\sqrt{i}+\sqrt{j}}{\sqrt{p^*_-}}$ ; $\rho(1^*|i,j,\omega_+)\approx \frac{\sqrt{i}\sqrt{j}}{p^*_+}$ ; $\rho(-1^*|i,j,\omega_-)\approx \frac{\sqrt{i}\sqrt{j}}{p^*_-}$.
\item $v(\omega_+)=v(\omega_-)=\frac{\sqrt{p^*_+}-\sqrt{p^*_-}}{\sqrt{p^*_+}+\sqrt{p^*_-}}=\frac{1-\sqrt{\frac{p^*_-}{p^*_+}}}{1+\sqrt{\frac{p^*_-}{p^*_+}}}$.
\end{itemize}

While these games are compact games, there are very specific ones since they are (up to a change of variables) mixed extensions of finite games. In particular the transitions functions are linear (up to a change of variables), and this is what entails the convergence of $v_\lambda$. A natural idea is to use the additional freedom in general compact games with interval action sets to construct a similar game such that $\rho(\omega_-|i,j,\omega_+) =\frac{\sqrt{i}+\sqrt{j}}{\sqrt{p^*_+(i,j)}}$ (where $p^*_+$ is no longer a constant but a function of $i$ and $j$), and  similar formulas for the other transitions.
If $\frac{p^*_-}{p^*_+}$ is slowly oscillating between two positive constants(which could not happen, by linearity, in the finite case), we expect that the value $v_\lambda$ also oscillates and thus does not converge. 

Because of this discussion, in the following we will only consider compact games played in pure (and not mixed) actions. This is very convenient since it yields easier computations. Of course in general there is no reason for the values $v_n$ and $v_\lambda$ to exist for a game played in pure actions; however in the following we show how to construct a game for which the values exist but do not converge.
\subsection{A class of compact games}

As the last section motivates us to do, let us consider the class $\mathcal{G}$ of compact stochastic games satisfying the following properties:
\begin{enumerate}
\item There are two nonabsorbing states $\omega_+$ and $\omega_-$, and two absorbing states $1^*$ and $-1^*$.
\item The action set of each player (denoted by $I$ and $J$ respectively) is the interval\footnote{For reasons that will become clear later (division by $1-\lambda$) it is better not to take $I=[0,1]$ but a smaller intervall.}  $\left[0,\frac{1}{16}\right]$.
\item The payoff depends only of the state: for all actions $i$ and $j$, $g(i,j,\omega_+)=g(i,j,1^*)=1$ and  $g(i,j,\omega_-)=g(i,j,-1^*)=-1$.
\item The transition probability $\rho$ is (jointly) continuous, and for all actions $i$ and $j$, $\rho(-1^*|i,j,\omega_+)=\rho(1^*|i,j,\omega_-)=0$.
\item In each nonabsorbing state and for each player, the pure action $\lambda$ is equalizing in the discounted game $\Gamma_\lambda$. That is, for each
$\lambda\in\left]0,\frac{1}{16}\right]$, and for each $i\in I$ and $j\in J$, the discounted value $v_\lambda$ satisfies
\begin{eqnarray}
\label{eq1} v_\lambda(\omega_+)&=&\lambda + (1-\lambda)\left[p^*_+(\lambda,j)+p_+(\lambda,j)v_\lambda(\omega_-)+(1-p^*_+(\lambda,j)-p_+(\lambda,j))v_\lambda(\omega_+)\right]\\
\label{eq2} v_\lambda(\omega_+)&=&\lambda + (1-\lambda)\left[p^*_+(i,\lambda)+p_+(i,\lambda)v_\lambda(\omega_-)+(1-p^*_+(i,\lambda)-p_+(i,\lambda))v_\lambda(\omega_+)\right]\\
\label{eq3} v_\lambda(\omega_-)&=&-\lambda + (1-\lambda)\left[-p^*_-(\lambda,j)+p_-(\lambda,j)v_\lambda(\omega_+)+(1-p^*_-(\lambda,j)-p_-(\lambda,j))v_\lambda(\omega_-)\right]\\
\label{eq4} v_\lambda(\omega_-)&=&-\lambda + (1-\lambda)\left[-p^*_-(i,\lambda)+p_-(i,\lambda)v_\lambda(\omega_+)+(1-p^*_-(i,\lambda)-p_-(i,\lambda))v_\lambda(\omega_-)\right].
\end{eqnarray}
\end{enumerate}

We remark that to define a game in $\mathcal{G}$ one only need to specify the four functions
\begin{eqnarray*}
p^*_+(i,j)&:=&\rho(1^*|i,j,\omega_+)\\
p_+(i,j)&:=&\rho(\omega_-|i,j,\omega_+)\\
p^*_-(i,j)&:=&\rho(-1^*|i,j,\omega_-)\\
p_-(i,j)&:=&\rho(\omega_+|i,j,\omega_-)
\end{eqnarray*}
since necessarily $\rho(\omega_+|i,j,\omega_+)=1-p^*_+(i,j)-p_+(i,j)$ and $\rho(\omega_-|i,j,\omega_-)=1-p^*_-(i,j)-p_-(i,j)$.

Also we observe that equations (\ref{eq1}) to  (\ref{eq4}) are characterizations of $v_\lambda$: any function $w_\lambda:\{\omega_+,\omega_-\}\to \mathds{R}$ satisfying the same system must be the discounted value of the game according to Lemma \ref{equalizinglemma}. Also remark that it implies that the discounted games have a value in pure strategies.

We first establish Theorem \ref{maintheorem} for discounted values:
\begin{theorem}\label{contreexemple}
There exists a game in $\mathcal{G}$ such that $v_\lambda$ does not converge as $\lambda$ goes to 0.
\end{theorem}

The idea of the construction of such an example is to think of the family $\{v_\lambda\}_{\lambda\in \left]0,\frac{1}{16}\right]}$ as a parameter of the game, and of the transition functions as unknowns, rather than the opposite. The construction is done in three steps: first, for any family $v_\lambda$ we identify good candidates $p^*_+$, $p^*_-$, $p_+$, $p_-$ that may lead to value $v_\lambda$ in $\Gamma_\lambda$. These candidate functions are in general neither in $[0,1]$ nor continuous ; but in a second step we show that when it is the case they indeed define a game in $\mathcal{G}$ with value $v_\lambda$. Finally, we find a family $v_\lambda$ that does not converge as $\lambda$ goes to 0, but such that the constructed candidates $p^*_+$, $p^*_-$, $p_+$, $p_-$ have the required regularity.

So let us fix a family $v_\lambda$ and try to find suitable functions $p^*_+$, $p^*_-$, $p_+$, and $p_-$.
By simplifying a bit equations (\ref{eq1}) to (\ref{eq4}), and replacing $\lambda$ by $\mu$ in (\ref{eq2}) and  (\ref{eq4}) one gets the following system (where $\lambda$ and $\mu$ are in $\left]0,\frac{1}{16}\right]$ while $i$ and $j$ are in $\left[0,\frac{1}{16}\right]$):

\begin{eqnarray}
\label{equationplus}v_\lambda(\omega_+)&=&\frac{\lambda + (1-\lambda)\left[p^*_+(\lambda,j)+p_+(\lambda,j)v_\lambda(\omega_-)\right]}{\lambda + (1-\lambda)p^*_+(\lambda,j)+(1-\lambda)p_+(\lambda,j)}\\
\label{equationplus2} v_\mu(\omega_+)&=&\frac{\mu + (1-\mu)\left[p^*_+(i,\mu)+p_+(i,\mu)v_\mu(\omega_-)\right]}{\mu + (1-\mu)p^*_+(i,\mu)+(1-\mu)p_+(i,\mu)}\\
\label{equationmoins}v_\lambda(\omega_-)&=&\frac{-\lambda + (1-\lambda)\left[-p^*_-(\lambda,j)+p_-(\lambda,j)v_\lambda(\omega_+)\right]}{\lambda + (1-\lambda)p^*_-(\lambda,j)+(1-\lambda)p_-(\lambda,j)}\\
\label{equationmoins2} v_\mu(\omega_-)&=&\frac{-\mu + (1-\mu)\left[-p^*_-(i,\mu)+p_-(i,\mu)v_\mu(\omega_+)\right]}{\mu + (1-\mu)p^*_-(i,\mu)+(1-\mu)p_-(i,\mu)}
\end{eqnarray}

In particular taking $j=\mu$ in (\ref{equationplus}) and $i=\lambda$ in (\ref{equationplus2}) one gets, for each couple $\lambda,\mu$ in $\left]0,\frac{1}{16}\right]$, the system

\[
\begin{cases}
v_\lambda(\omega_+)&=\frac{\lambda + (1-\lambda)\left[p^*_+(\lambda,\mu)+p_+(\lambda,\mu)v_\lambda(\omega_-)\right]}{\lambda + (1-\lambda)p^*_+(\lambda,\mu)+(1-\lambda)p_+(\lambda,\mu)}\\
v_\mu(\omega_+)&=\frac{\mu + (1-\mu)\left[p^*_+(\lambda,\mu)+p_+(\lambda,\mu)v_\mu(\omega_-)\right]}{\mu + (1-\mu)p^*_+(\lambda,\mu)+(1-\mu)p_+(\lambda,\mu)}.
\end{cases}
\]

It is convenient to denote $s(\lambda)=\frac{v_\lambda(\omega_+)+v_\lambda(\omega_-)}{2}$ and $d(\lambda)=\frac{v_\lambda(\omega_+)-v_\lambda(\omega_-)}{2}$, so the system becomes

\[
\begin{cases}
 (1-\lambda)(s(\lambda)+d(\lambda)-1)p^*_+(\lambda,\mu) +2(1-\lambda)d(\lambda) p_+(\lambda,\mu) &=\lambda (1-s(\lambda)-d(\lambda))\\
 (1-\mu)(s(\mu)+d(\mu)-1)p^*_+(\lambda,\mu) +2(1-\mu)d(\mu) p_+(\lambda,\mu) &=\mu (1-s(\mu)-d(\mu))\\\end{cases}.
\]

When $\lambda\neq \mu$ the unique solution (assuming for a moment that the system is not degenerate) is given by

 \begin{eqnarray}
\label{eqdef1}p_+(\lambda,\mu)& =& \frac{(\lambda-\mu)(1-s(\lambda)-d(\lambda))(1-s(\mu)-d(\mu))}{2(1-\lambda)(1-\mu)[d(\lambda)(1-s(\mu))-d(\mu)(1-s(\lambda))]}\\
p^*_+(\lambda,\mu) &=& \frac{\lambda(1-\mu) d(\mu)(1-s(\lambda)-d(\lambda))-\mu(1-\lambda) d(\lambda)(1-s(\mu)-d(\mu))}  {(1-\label{eqdef2}\lambda)(1-\mu)[d(\lambda)(1-s(\mu))-d(\mu)(1-s(\lambda))]}
\end{eqnarray}

Similarly, considering equations  (\ref{equationmoins}) and (\ref{equationmoins2}) yields, for $\lambda\neq\mu$ in $\left]0,\frac{1}{16}\right]$

 \begin{eqnarray}
\label{eqdef3}p_-(\lambda,\mu)& =& \frac{(\lambda-\mu)(1+s(\lambda)-d(\lambda))(1+s(\mu)-d(\mu))}{2(1-\lambda)(1-\mu)[d(\lambda)(1+s(\mu))-d(\mu)(1+s(\lambda))]}\\
\label{eqdef4}p^*_-(\lambda,\mu) &=& \frac{\lambda(1-\mu) d(\mu)(1+s(\lambda)-d(\lambda))-\mu (1-\lambda) d(\lambda)(1+s(\mu)-d(\mu))}  {(1-\lambda)(1-\mu)[d(\lambda)(1+s(\mu))-d(\mu)(1+s(\lambda))]}
\end{eqnarray}

In general there is no guarantee that the functions defined by equations (\ref{eqdef1}) to (\ref{eqdef4}) will be positive, continuously extendable, or even well defined. However we now show that when they are, they define a game in $\mathcal{G}$.
\begin{defi}
A pair $(s, d)$ of continuous functions from $]0,\frac{1}{16}]$ to $\mathds{R}$ is \emph{feasible} if there exists a game in $\mathcal{G}$ such that
\begin{eqnarray*}
v_\lambda(\omega_+)&=&s(\lambda)+d(\lambda)\\
v_\lambda(\omega_-)&=&s(\lambda)-d(\lambda).
\end{eqnarray*}
\end{defi}

\begin{lem}\label{lemmedefp}
Assume that for $\lambda\neq\mu$ in $\left]0,\frac{1}{16}\right]$, the quantities defined in equations (\ref{eqdef1}) to (\ref{eqdef4}) are well defined, with value in $\left[0,\frac{1}{2}\right]$. Also assume that the four functions can be continuously extended to $\left[0,\frac{1}{16}\right]^2$. Then $(s, d)$ is feasible.
\end{lem}

\begin{proof}
Let $\Gamma$ be the stochastic game satisfying assumptions a) to d) and with transitions functions defined by equations (\ref{eqdef1}) to (\ref{eqdef4}) (and their continuous extensions), by $\rho(\omega_+|i,j,\omega_+)=1-p^*_+(i,j)-p_+(i,j)\in[0,1]$ and by $\rho(\omega_-|i,j,\omega_-)=1-p^*_-(i,j)-p_-(i,j)\in[0,1]$. It remains to show that assumption e) is satisfied, with $v_\lambda(\omega_+)=s(\lambda)+d(\lambda)$ and $v_\lambda(\omega_-)=s(\lambda)-d(\lambda)$. By construction, for every discount factor $\lambda$ equations (\ref{eq1}) to (\ref{eq4}) are satisfied for $i$ and $j$ in $\left]0,\frac{1}{16}\right]\backslash\{\lambda\}$, and so by continuity they are satisfied for $i$ and $j$ in $\left[0,\frac{1}{16}\right]$.
\end{proof}

%
%

\subsection{Construction of a specific counterexample}\label{sectiontechnique}
To establish Theorem \ref{contreexemple} it is thus enough to find a couple $(s, d)$ such that the assumptions of Lemma \ref{lemmedefp} are satisfied but $s(\lambda)\pm d(\lambda)$ does not converge as $\lambda$ goes to 0. We first give an intuition leading to our choice of specific $d$ and $s$. 

Let $(s,d)$ be any feasible couple. Then, for the values not to converge, it is necessary that $d(\lambda)$ tends slowly to 0 as $\lambda$ goes to 0, for the following reasons.
\begin{itemize}
\item Let $v_1$ and $v_2$ be any two accumulation points of $v_\lambda$ such that  $\max_\omega\{v_1(\omega)-v_2(\omega)\}>0$. Define $\Omega_1=\Argmax_\omega\{v_1(\omega)-v_2(\omega)\}$ and $\Omega_2=\Argmax_{\omega\in\Omega_1}\{v_1(\omega)\}$. Reasoning as in \cite{SoVi} yields to a contradiction as soon as $\Omega_2$ is a singleton ; this implies that $d(\lambda)$ goes to 0 as $\lambda$ goes to 0.
\item Assume for example that $d(\lambda)=0$ for $\lambda$ small enough. Then $v_\lambda(\omega_+)=v_\lambda(\omega_-)$, hence the values won't change if we replace any transition from $\omega^+$ to $\omega^-$ by a transition from $\omega^+$ to $\omega^+$ ; and any transition from $\omega^-$ to $\omega^+$ by a transition from $\omega^-$ to $\omega^-$. But the resulting game is just two absorbing games played in parallel, and absorbing games have an asymptotic value, a contradiction. If $d(\lambda)= o(\lambda)$, the values won't change "much" in the auxiliary game, and the contradiction is the same.
\end{itemize}

Denote by $\sqrt{\phantom{x}}$ the function $x\to\sqrt{x}$. Because of the reasons stated above, in this section we fix $d=\sqrt{\phantom{x}}$. Since the payoff function is bounded, it is easy to see that if $(s,\sqrt{\phantom{x}})$ is feasible and $s$ is continuously differentiable, then $s$ and $\lambda s'(\lambda)$ are bounded. We now prove a reciprocal:

\begin{proposition}\label{mainprop}
Let\footnote{We denote ${C}^1(A,B)$ the set of continuously differentiable functions from $A$ to $B$.} $s\in {C}^1(]0,\frac{1}{16}],\mathds{R})$. Assume that $s$ and $x\rightarrow x s'(x)$ are both bounded by $\frac{1}{16}$. Then $\left(s, \sqrt{\phantom{x}}\right)$ is feasible.
\end{proposition}

Proposition \ref{contreexemple} is an immediate consequence since there are functions $s(x)$ satisfying the assumptions of Proposition \ref{mainprop} but without a limit as $x$ goes to 0. Take for example $s(x)=\frac{\sin \ln x}{16}$.

We start by a technical lemma.
\begin{lem}\label{basic}
Let $s\in {C}^1(]0,\frac{1}{16}],\mathds{R})$. Assume that $s$ and $x\rightarrow x s'(x)$ are both bounded by $C$.
Then the two functions defined on $]0, \frac{1}{16}[^2$ to $\mathds{R}$ by
\[
f_1(x,y)=\begin{cases}
\frac{\sqrt{x} s(x)-\sqrt{y} s(y)}{\sqrt{x}-\sqrt{y}}& \text{if } x\neq y\\
2xs'(x)+s(x)  & \text{if } x=y
\end{cases}
\]
and
\[
f_2(x,y)=\begin{cases}
\frac{\sqrt{y} s(x)-\sqrt{x} s(y)}{\sqrt{x}-\sqrt{y}}& \text{if } x\neq y\\
2xs'(x)-s(x)  & \text{if } x=y
\end{cases}
\]

are jointly continuous and bounded by $3C$.
\end{lem}
We stress out that we do not need $s$ to have a limit, as $x$ goes to 0, for this lemma to hold (and, in fact, this is precisely what will allow us to construct our counterexample).

\begin{proof}
For $x\neq y$, $f_1(x,y)=(\sqrt{x}+\sqrt{y})\frac{\sqrt{x} s(x)-\sqrt{y} s(y)}{x-y}$, hence the mean value theorem ensures that $f_1$ is continuous. Moreover, for $y<x$,
\begin{eqnarray*}
|f_1(x,y)|&\leq&\frac{1}{\sqrt{x}-\sqrt{y}}\int_y^x \left| \left(\sqrt{z} s(z)\right)^{'} \right|dz \\
&=&\frac{1}{\sqrt{x}-\sqrt{y}}\int_y^x \left| \sqrt{z}s'(z)+\frac{s(z)}{2\sqrt{z}}  \right| dz\\
&\leq& \frac{1}{\sqrt{x}-\sqrt{y}}\int_y^x \frac{3Cdz}{2\sqrt{z}}\\
&=&3C.
\end{eqnarray*}

For $x\neq y$, $f_2(x,y)=(\sqrt{x}+\sqrt{y})\sqrt{xy}\frac{\frac{ s(x)}{\sqrt{x}}-\frac{ s(y)}{\sqrt{y}}}{x-y}$, hence the mean value theorem ensures that $f_2$ is continuous. Moreover, for $y<x$,
\begin{eqnarray*}
|f_2(x,y)|&\leq&\frac{\sqrt{xy}}{\sqrt{x}-\sqrt{y}}\int_y^x \left| \left( \frac{s(z)}{\sqrt{z}}\right)^{'}  \right|dz \\
&=&\frac{\sqrt{xy}}{\sqrt{x}-\sqrt{y}}\int_y^x \left| \frac{s'(z)}{\sqrt{z}}-\frac{s(z)}{2z\sqrt{z}}  \right| dz\\
&\leq& \frac{\sqrt{xy}}{\sqrt{x}-\sqrt{y}}\int_y^x \frac{3Cdz}{2z\sqrt{z}}\\
&=&3C.
\end{eqnarray*}
\end{proof}

%
%
%
%
%

For $d=\sqrt{\phantom{x}}$, we remark that the quantities defined in (\ref{eqdef1}) to (\ref{eqdef4}) can be rewritten as,
for $\lambda\neq \mu $ in $\left]0,\frac{1}{16}\right]$,

\begin{eqnarray}
p_+(\lambda,\mu)& =&\frac{(\sqrt{\lambda}+\sqrt{\mu})(1-\sqrt{\lambda}-s(\lambda))(1-\sqrt{\mu}-s(\mu))}{2(1-\lambda)(1-\mu)(1+f_2(\lambda,\mu))}  \\
p^*_+(\lambda,\mu) &=&\frac{\sqrt{\lambda\mu}\left[(1-\sqrt{\lambda})(1-\sqrt{\mu})-f_1(\lambda,\mu)+\sqrt{\lambda\mu}f_2(\lambda,\mu)\right]}{(1-\lambda)(1-\mu)(1+f_2(\lambda,\mu))} \\
p_-(\lambda,\mu)& =&\frac{(\sqrt{\lambda}+\sqrt{\mu})(1-\sqrt{\lambda}+s(\lambda))(1-\sqrt{\mu}+s(\mu)}{2(1-\lambda)(1-\mu)(1-f_2(\lambda,\mu))} \\
p^*_-(\lambda,\mu) &=& \frac{\sqrt{\lambda\mu}\left[(1-\sqrt{\lambda})(1-\sqrt{\mu})+f_1(\lambda,\mu)-\sqrt{\lambda\mu}f_2(\lambda,\mu)\right]}{(1-\lambda)(1-\mu)(1-f_2(\lambda,\mu))}.
\end{eqnarray}

The four following lemmas establish that the regularity conditions in Lemma \ref{lemmedefp} are satisfied under the assumptions of Proposition \ref{mainprop}.
\begin{lem}\label{contpstarplus}
Let $s\in {C}^1(]0,\frac{1}{16}],\mathds{R})$. Assume that $s$ and $x\rightarrow x s'(x)$ are both bounded by $\frac{1}{16}$. Then
the function defined on $\left[0,\frac{1}{16}\right]^2$ by

\[
p^*_+(x,y)=\begin{cases}
\frac{\sqrt{xy}\left[(1-\sqrt{x})(1-\sqrt{y})-f_1(x,y)+\sqrt{xy}f_2(x,y)\right]}{(1-x)(1-y)(1+f_2(x,y))} & \text{if } xy>0\\
 0   & \text{if } xy=0
\end{cases}
\]
is well defined, jointly continuous, with value in $\left[0,\frac{1}{2}\right]$.
\end{lem}

\begin{proof}
Lemma \ref{basic} implies that the denominator is positive when $xy>0$, hence $p^*_+$ is well defined. The same lemma also implies that $p^*_+$ is jointly continuous on $]0,\frac{1}{16}]^2$. Finally, the bounds on $f_1$ and $f_2$ and the fact that $x$ and $y$ are less than $\frac{1}{16}$ imply that
\[
\sqrt{xy} \frac{9/16-3/16-3/256}{19/16}\leq p^*_+(x,y)\leq \sqrt{xy} \frac{1+3/16+3/256}{(15/16)^2\times13/16}=\frac{4912}{2 925}\sqrt{xy} <\frac{1}{2}
\]
hence $p^*_+$ is also jointly continuous at any $(x,y)$ with $xy=0$, and takes its value in $\left[0,\frac{1}{2}\right]$.
\end{proof}

\begin{lem}
Let $s\in {C}^1(]0,\frac{1}{16}],\mathds{R})$. Assume that $s$ and $x\rightarrow x s'(x)$ are both bounded by $\frac{1}{16}$. Then
the function defined on $\left[0,\frac{1}{16}\right]^2$ by

\[
p^*_-(x,y)=\begin{cases}
\frac{\sqrt{xy}\left[(1-\sqrt{x})(1-\sqrt{y})+f_1(x,y)-\sqrt{xy}f_2(x,y)\right]}{(1-x)(1-y)(1-f_2(x,y))} & \text{if } xy>0\\
 0   & \text{if } xy=0
\end{cases}
\]
is well defined, jointly continuous, with value in $\left[0,\frac{1}{2}\right]$.
\end{lem}

\begin{proof}
Same as the previous proof, replacing $s$ by its opposite.
\end{proof}

\begin{lem}
Let $s\in {C}^1(]0,\frac{1}{16}],\mathds{R})$. Assume that $s$ and $x\rightarrow x s'(x)$ are both bounded by $\frac{1}{16}$. Then
the function defined on $\left[0,\frac{1}{16}\right]^2$ by

\[
p_+(x,y)=\begin{cases}
\frac{(\sqrt{x}+\sqrt{y})(1-\sqrt{x}-s(x))(1-\sqrt{y}-s(y))}{2(1-x)(1-y)(1+f_2(x,y))} & \text{if } xy>0\\
\frac{\sqrt{x}(1-\sqrt{x}-s(x))}{2(1-x)} & \text{if } x>0 \text{ and } y=0\\
\frac{\sqrt{y}(1-\sqrt{y}-s(y))}{2(1-y)} & \text{if } y>0 \text{ and } x=0\\
 0   & \text{if } x=y=0
\end{cases}
\]
is well defined, jointly continuous, with value in $\left[0,\frac{1}{2}\right]$.
\end{lem}

\begin{proof}
Lemma \ref{basic} implies that the denominator is positive when $xy>0$, hence $p_+$ is well defined. The same lemma also implies that $p_+$ is jointly continuous on $]0,\frac{1}{16}]^2$.

Remarking that for $x\neq y$ , $1+f_2(x,y)=\frac{1}{\sqrt{x}-\sqrt{y}} (\sqrt{x}(1-\sqrt{y}-s(y))-\sqrt{y}(1-\sqrt{x}-s(x)))$ one gets

\[
p_+(x,y)=\frac{x-y}{2(1-x)(1-y)\left[ \frac{\sqrt{x}} {1-\sqrt{x}-s(x)}-\frac{\sqrt{y}} {1-\sqrt{y}-s(y)}   \right]}
\]

hence the joint continuity of $p_+$ at any point where $xy=0$ and $x+y>0$.

Finally, the bounds on $f_1$ and $f_2$ and the fact that $x$ and $y$ are less than $\frac{1}{16}$ imply that
\[
(\sqrt{x}+\sqrt{y}) \frac{(11/16)^2}{2\times19/16}\leq p_+(x,y)\leq (\sqrt{x}+\sqrt{y}) \frac{(17/16)^2}{2\times(15/16)^2\times13/16}=\frac{2312}{2 925} (\sqrt{x}+\sqrt{y})<\frac{1}{2}
\]
hence $p_+$ is also jointly continuous at $(0,0)$, and takes its value in $\left[0,\frac{1}{2}\right]$.
\end{proof}

\begin{lem}\label{contpmoins}
Let $s\in {C}^1(]0,\frac{1}{16}],\mathds{R})$. Assume that $s$ and $x\rightarrow x s'(x)$ are both bounded by $\frac{1}{16}$. Then
the function defined on $\left[0,\frac{1}{16}\right]^2$ by

\[
p_-(x,y)=\begin{cases}
\frac{(\sqrt{x}+\sqrt{y})(1-\sqrt{x}+s(x))(1-\sqrt{y}+s(y))}{2(1-x)(1-y)(1-f_2(x,y))} & \text{if } xy>0\\
\frac{\sqrt{x}(1-\sqrt{x}+s(x))}{2(1-x)} & \text{if } x>0 \text{ and } y=0\\
\frac{\sqrt{y}(1-\sqrt{y}+s(y))}{2(1-y)} & \text{if } y>0 \text{ and } x=0\\
 0   & \text{if } x=y=0
\end{cases}
\]
is well defined, jointly continuous, with value in $\left[0,\frac{1}{2}\right]$.
\end{lem}

\begin{proof}
Same as the previous proof, replacing the function $s$ by its opposite.
\end{proof}

\begin{proof}[Proof of Proposition \ref{mainprop}]
It is an immediate consequence of Lemma \ref{lemmedefp}, since by the four preceding lemmas the functions defined by equations (\ref{eqdef1}) to (\ref{eqdef4}) have all the required properties.
\end{proof}


\subsection{The case of finitely repeated stochastic game.}
In this section we construct examples where $v_n$ does not converge as $n$ goes to infinity. The idea is to construct an example in which $v_\lambda$ does not converge and such that the sequence $v_n$ has the same asymptotic behavior as $v_\lambda$. The following lemma is a slight variation of a result of Neyman \cite{Ne}.

\begin{lem}\label{lemvlambdavn}
Let $\Gamma$ be any stochastic game. Assume that $v_\lambda$ is of class $\mathcal{C}^1$, and that for all $\omega$, $\frac{\mathrm{d} v_\lambda(\omega)}{\mathrm{d}\lambda}=o(\frac{1}{\lambda})$. Then $v_n$ and $v_\lambda$ have the same accumulation points.
\end{lem}

\begin{proof}
Denote $w_n=v_\lambda$ for $\lambda=\frac{1}{n}$. By assumptions and the mean value theorem, \[\sup_{\mu\in[\frac{1}{n},\frac{1}{n-1}]}\|w_n-v_\mu\|_\infty=o(\frac{1}{n}).\] Hence $v_\lambda$ and $w_n$ have the same accumulation points.

By a argument due to Neyman(Theorem 4 in \cite{Ne}),
\[
\|w_n-v_n\|_\infty \leq \frac{1}{n}\sum_{i=1}^{n-1} i\|w_{i+1}-w_i\|_\infty.
\]

So we only need to observe that, by the mean value theorem,
\begin{eqnarray*}
i\|w_{i+1}-w_i\|_\infty&=&i\|v_{\frac{1}{i+1}}-v_{\frac{1}{i}}\|_\infty\\
&\leq&\frac{1}{i+1} \sup_{\lambda\in[1/i,1/i+1]} \left\|\frac{\mathrm{d} v_\lambda}{\mathrm{d}\lambda}\right\|_\infty\\
&=&o(1).
\end{eqnarray*}
 \end{proof}

We can now prove Theorem \ref{maintheorem}.

\begin{proof}[Proof of Theorem \ref{maintheorem}]
Let\footnote{The function $\frac{\sin{\ln x}}{16}$ used previously would not work here since its derivative is not a $o(1/x)$.}   $s(x)=\frac{\sin\ln (-\ln x)}{16}$ and $d(x)=\sqrt{x}$. Since $s'(x)=-\frac{\cos\ln (-\ln x)}{16x\ln x}$, $xs'(x)$ is bounded by $\frac{1}{64\ln(2)}<\frac{1}{16}$, and goes to 0 as $x$ goes to 0. By Proposition \ref{mainprop}, $(s,d)$ is feasible but $v_\lambda$ does not converge as $\lambda$ goes to 0. By Lemma \ref{lemvlambdavn}, $v_n$ does not converge as $n$ goes to infinity.
\end{proof}

 \begin{remark}
In fact, we could prove, exactly in the same way, that for the example constructed in the last proof, any admissible sequence $z_n$ (as defined in \cite{So}) also diverges as $n$ goes to infinity.
\end{remark}

\section{Concluding remarks and open problems}\label{sectionconclusion}

We first point out that these examples are minimal in several aspects:
\begin{itemize}
\item There are only two nonabsorbing states. Compact games with only one nonabsorbing state (called absorbing games) have an asymptotic \cite{RoSo} value.
\item In nonabsorbing states, the payoff does not depend on actions. If the payoff also did not depend on the current nonabsorbing state, the game would be a compact recursive game and would have an asymptotic value \cite{So4}.
\item There are exactly two initial states $\omega$ (the two absorbing states) such that $v_n(\omega)$ converges. For every compact stochastic game, there are at least two initial states such that $v_n(\omega)$ converges \cite{KoNe, Ne}.
\item Also remark that the action sets are not general compact sets but rather real intervals, and that all discounted games $\Gamma_\lambda$ have values in pure strategies.
\end{itemize}

At the beginning of Section \ref{sectiontechnique} we gave necessary conditions for $(s,d)$ to be feasible. In fact, quite surprisingly, it turns out that those necessary conditions are almost sufficient. Explicitly, one can prove, using the techniques presented in Section \ref{sectiontechnique}, that:

\begin{proposition}
Let $s$ and $d$ be two continuously differentiable functions from $]0,\frac{1}{16}]$ to $\mathds{R}$. Assume that
\begin{itemize}
\item $s$ and $\lambda s'(\lambda)$ are bounded.
\item $d$ is nonnegative and there exists $\varepsilon>0$ such that\footnote{In particular, any function $\lambda^\alpha$ for $\alpha\in]0,1[$ satisfies this condition.} for all $\lambda\in]0,1]$,
\[
\varepsilon\leq \frac{\lambda d'(\lambda)}{d(\lambda)}\leq 1-\varepsilon.
\]
\end{itemize}
Then $(As, Bd)$ is feasible for any nonnegative constants $A$ and $B$ small enough.
\end{proposition}

Let us now briefly discuss the regularity of the transitions functions in our counterexamples. One may remark that while these functions are constructed to be continuous, they are not continuously differentiable in 0, nor even Lipschitz-continuous. However, we affirm that this lack of regularity is not at all the reason of the divergence of $v_\lambda$. Indeed, for any nonnegative $r$, replacing $x$ and $y$ by $x^r$ and $y^r$ in the definition of the transition functions will not change the values (it is just a relabeling of actions) while it will regularize the transition functions. To say it another way, we only considered games where the pure action $\lambda$ is optimal in $\Gamma_\lambda$, but this is just to make calculations easier, and relaxing this assumption we can construct transitions as regular as one wants.

Rather than the regularity of the transition functions, we argue that the issue here is their infinite number of oscillations. Recall that in our counterexample, $p_+(i,j)$ is of the order of $\sqrt{i}+\sqrt{j}$, and $p_+^*(i,j)$ is of the order of $\sqrt{ij}$. Thus in $\Gamma_\lambda$, starting from $\omega_+$, Player 2 should play neither a high $j$ (otherwise Player 1 may absorb with payoff 1 and high probability) nor a low $j$ (otherwise Player 1 may stay in $\omega_+$ with payoff 1 until the game is essentially finished). Hence he should play the intermediate action $j=\lambda$, and the same thing is true in $\omega_-$ and for Player 1. So, under optimal play, the order of magnitude of the time between two transitions from a nonabsorbing state to the other is $\lambda^{-\frac{1}{2}}$. Hence, after the first $\lambda^{-\frac{2}{3}}$ stages (during which the accumulated discounted payoff has been negligible), there still has been no absorption with probability almost 1, and the occupation measure has almost reached the invariant measure $\frac{p_-(\lambda,\lambda)}{p_-(\lambda,\lambda) + p_+(\lambda,\lambda)}\cdot \omega_+ + \frac{p_+(\lambda,\lambda)}{p_-(\lambda,\lambda) + p_+(\lambda,\lambda)}\cdot \omega_-$. We have thus established that the discounted payoff given nonabsorption is approximately $\frac{p_-(\lambda,\lambda) - p_+(\lambda,\lambda)}{p_-(\lambda,\lambda) + p_+(\lambda,\lambda)}$ ; similarly one sees that the expected payoff after absorption is function of the relative importance of $p_+,p_-,p_+^*$ and $p_-^*$. Oscillations of the ratios between these quantities thus imply oscillations of the discounted payoff under optimal play, which is $v_\lambda$. For a compact game that is the mixed extension (up to relabelling of actions) of some finite games, these quantities cannot oscillate infinitely often, and we think this is the reason why the values converge in the finite framework. Since semi-algebraic functions cannot oscillate infinitely often, the following natural question is rather intriguing:

\begin{itemize}
\item[(i)] Let $\Gamma$ be a compact game with semi-algebraic payoff and transition functions. Is there an asymptotic value ?
\end{itemize}

When there is only one player, while there may be no 0-optimal play in the infinite game \cite{Ba1}, the asymptotic value always exists for compact games \cite{Ba3,DyYu}. In fact it exists with no hypotheses at all on the action set as long as the number of states is finite \cite{Ren3}. When there are two players, the asymptotic value exists for games with finitely many actions for each player \cite{BeKo}, but we showed that asymptotic value may not exist for compact games. This leads to this question in an intermediate setting:

\begin{itemize}
\item[(ii)] Let $\Gamma$ be a compact game in which Player 1 has a finite number of actions. Is there an asymptotic value ?
\end{itemize}

Finally, the asymptotic value exists for compact games in which Player 2 has no influence on the transition, and in fact even in a more general setting in which Player 2 is also not perfectly informed of the state \cite{Ren2}. In our construction it is important that the transitions are jointly controlled. From $\omega_+$ Player 1 cannot ensure to go to $1^*$ with positive probability, while Player 2 can force a transition to $\omega_-$ with high probability. However Player 2 cannot at the same time prevent any transition to $1^*$ and ensure a positive probability to go to $\omega_-$. Hence one may wonder:

\begin{itemize}
\item[(iii)] Let $\Gamma$ be a compact game in which each state is controlled by one player (but different states may be controlled by different players). Is there an asymptotic value ?
\end{itemize}

The answer to these three questions is not know in general, however the particular case of semi algebraic games which also satisfy either condition (ii) or (iii) is settled (with positive answer) in \cite{BoGaVi}.

A last remark is that there is a huge gap between compact games with one and two nonabsorbing states. We just showed that there is no asymptotic value for games with two nonabsorbing states ; while for one nonabsorbing states the stronger notion of uniform value (when the payoffs are observed) also holds \cite{MNR}. In fact it does not seem easy to construct a compact game with an asymptotic but no uniform value (when the payoffs are observed).

\section*{Acknowledgments}
This paper owes a lot to Sylvain Sorin. I pleasantly remember countless discussions about compact games and why they should have an asymptotic value or not, as well as devising with him a number of "almost proofs" of convergence. This was decisive to understand the right direction to go to stumble upon this counterexample.

I also would like to thank J\'erome Bolte for being the first to warn me about nonsemialgebraic functions, J\'er\^ome Renault for raising several interesting questions while I was writing this paper, as well as Andrzej S. Nowak for useful references.

\end{document}